\documentclass[12pt,a4paper]{amsart}
\usepackage{mathrsfs}
\usepackage{color}
\usepackage{ulem}
\newtheorem{theo+}              {Theorem}           [section]
\newtheorem{prop+}  [theo+]     {Proposition}
\newtheorem{coro+}  [theo+]     {Corollary}
\newtheorem{lemm+}  [theo+]     {Lemma}
\newtheorem{exam+}  [theo+]     {Example}
\newtheorem{rema+}  [theo+]     {Remark}
\newtheorem{defi+}  [theo+]     {Definition}

\newenvironment{theorem}{\begin{theo+}}{\end{theo+}}
\newenvironment{proposition}{\begin{prop+}}{\end{prop+}}
\newenvironment{corollary}{\begin{coro+}}{\end{coro+}}
\newenvironment{lemma}{\begin{lemm+}}{\end{lemm+}}

\usepackage{amsthm}
\theoremstyle{plain} \theoremstyle{remark}
\newtheorem{remark}{Remark}

\def \r{\mbox{${\mathbb R}$}}
\def\E{/\kern-1.0em \equiv }

\title{Biharmonic hypersurfaces in product spaces}
\author{Yu Fu$^{*}$, Shun Maeta$^{**}$  and Ye-Lin Ou$^{***}$ }

\address{School of Mathematics, \newline\indent Dongbei University of Finance and Economics,
\newline\indent Dalian 116025,
\newline\indent China.
\newline\indent
E-mail:yufu@dufe.edu.cn and yufudufe@gmail.com}

\address{Department of Mathematics, \newline\indent Shimane University,\newline\indent Matsue, 690-8504,
\newline\indent Japan.
\newline\indent
E-mail:shun.maeta@gmail.com and maeta@riko.shimane-u.ac.jp}

\thanks{
\noindent $^{*}$The first named author is supported by NSFC (No. 11601068).
\newline\indent $^{**}$The second named author is supported partially by the Grant-in-Aid for Young Scientists(B),
 No.15K17542, Japan Society for the Promotion of Science, and  partially by JSPS Overseas Research
Fellowships 2017-2019 No.70. The work was done while he was visiting
the Department of Mathematics of Texas A $\&$ M University-Commerce
as a Visiting Scholar and he is grateful to the department and the
university for the hospitality he had received during the visit.
\newline\indent $^{***}$ The third named author is supported by a grant
from the Simons Foundation ($\#427231$, Ye-Lin Ou)}

\address{Department of
Mathematics,\newline\indent Texas A $\&$ M University-Commerce,
\newline\indent Commerce TX 75429,\newline\indent USA.\newline\indent
E-mail:yelin.ou@tamuc.edu}
\begin{document}

\title[Biharmonic hypersurfaces in product spaces]{Biharmonic hypersurfaces in a product space $L^m\times \mathbb{R}$}
\subjclass[2010]{58E20, 53C12} \keywords{Biharmonic hypersurfaces,
angle function, product spaces, rotation hypersurfaces, totally umbilical hypersurfaces.}
\date{06/04/2019}
 \maketitle

\section*{Abstract}
\begin{quote}
{\footnotesize In this paper, we study biharmonic hypersurfaces in a
product $L^m\times \mathbb{R}$ of an Einstein space $L^m$ and a real
line $\mathbb R$. We prove that a biharmonic hypersurface  with
constant mean curvature in such a product  is either minimal or a vertical cylinder generalizing a result of \cite{OW} and \cite{FOR}. We
derived the  biharmonic equation for hypersurfaces in $S^m\times \mathbb{R}$ and $H^m\times \mathbb{R}$ in terms of the angle function of the hypersurface,  and use it to obtain some classifications of biharmonic hypersurfaces in such spaces. These include classifications of biharmonic hypersurfaces which are totally umbilical or semi-parallel for $m\ge 3$, and some classifications  of biharmonic surfaces  in $S^2\times \mathbb{R}$ and $H^2\times \mathbb{R}$ which are constant angle or belong to certain classes of rotation surfaces. }
\end{quote}

\section{Introduction }
The study of the geometry of the hypersurfaces in the conformally flat spaces $S^m\times \r$ and $H^m\times \r$ has been receiving a growing attention since 2002. It was initiated  by U. Abresch and H. Rosenberg in \cite{AR} and \cite{Ro} where they studied minimal and constant mean curvature surfaces in $S^2\times \r$ and $H^2\times \r$.

A  fundamental theorem for the existence of hypersurfaces in $S^m\times \r$ and $H^m\times \r$ was proved by B. Daniel in \cite{Da}.

The existence and some classifications of surfaces of constant Gauss curvature in $S^2\times \r$ and $H^2\times \r$ were studied in \cite{AEG} and \cite{AEG2} whilst for $m\ge 3$, \cite{MT} gave a complete classification of constant sectional curvature hypersurfaces in $S^m\times \r$ and $H^m\times \r$. An interesting consequence of the classification in \cite{MT} is that for $m\ge 4$, a constant sectional curvature hypersurface (even a local one) in $S^m\times \r$ (resp. $H^m\times \r$) has to be a rotation hypersurface with constant sectional curvature $c \ge1$ (resp, $c\ge -1)$, and for $m = 3$, there is exactly one class of nonrotational hypersurfaces of $S^3\times \r$ and $H^3\times \r$ with constant sectional curvature. Each such hypersurface in
this class in $S^3\times \r$ (resp. $H^3\times \r$) has constant sectional curvature  $c\in(0, 1)$ (resp. $c \in (-1, 0)$), and is constructed in an explicit way
by means of a family of parallel flat surfaces in $S^3$ (resp. $H^3$).

Classification of totally umbilical, parallel and semi-parallel
hypersurfaces of $H^m\times \r$ and $S^m\times \r$ were done in in
\cite{CKV} and \cite{VV}, respectively. An interesting consequence
of these classifications shows that, unlike the situation in space
form, a totally umbilical hypersurface $H^m\times \r$ or $S^m\times
\r$ may not be parallel. A complete classification of totally
umbilical submanifolds with any codimension in $S^m\times \r$ was
obtained by B. Mendonca and R. Tojeiro in \cite{MT14}.

 Here, we recall that a hypersurface with the second fundamental form $b$ is said to be {\it pseudo-parallel} if $R\cdot b=\phi(X\wedge Y)b$ for some real-valued function $\phi$ on the hypersurface, where $R\cdot b$ is a $(0,4)$-tensor
field defined by
\begin{align}\notag
 (R\cdot b)(X, Y, U, V) =& (R(X, Y)\cdot b)(U, V) \\\notag
 = &-b(R(X, Y)U, V) - b(U, R(X, Y)V).
\end{align}
A pseudo-parallel hypersurface with $\phi\equiv 0$ is said to be {\it semi-parallel}, i.e., it satisfies the condition $R\cdot b = 0$. Recalling that a hypersurface is {\it parallel} means $\nabla b = 0$, we clearly have the following inclusion relations:
\begin{center}{\scriptsize
$ \{ {\rm Parallel\; hypersurfaces}\} \subset  \{ {\rm Semi-parallel \;hypersurfaces}\} \subset \{ {\rm Pseudo-parallel\; hypersurfaces}\}$}.
\end{center}

Constant angle surfaces in $S^2\times\mathbb R$ and $H^2\times\mathbb R$ were studied and characterized in \cite{DFVV}, \cite{DM}, and \cite{DFV09, DMN}. Later, Tojeiro \cite{To} proved that a constant angle hypersurface in $S^m\times\mathbb R$ or $H^m\times\mathbb R$ has to be a slice, a vertical cylinder, or a hypersurface that can be parametrized explicitly by using the parametrization of a semi-parallel hypersurface in the first factor with a linear parametrization in the second factor.

 Rotation hypersurfaces in $S^m\times\mathbb R$ and $H^m\times\mathbb R$ ware  introduced and studied in \cite{DFV}  where the authors classified minimal rotation hypersurfaces, and intrinsically flat rotation hypersurfaces in $S^m\times\mathbb R$ and $H^m\times\mathbb R$. For rotation surfaces with constant Gauss curvature in $S^2\times\mathbb R$ and $H^2\times\mathbb R$ see \cite{AEG, AEG2}.

For classifications of pseudo-parallel hypersurfaces in $S^m\times \r$ and $H^m\times \r$  see \cite{LY} and  \cite{LT}.
It was proved in \cite{LY} and \cite{To} that the hypersurfaces of $S^m\times \r$ and $H^m\times \r$ that have exactly three principal curvatures are vertical cylinders over a semi-parallel hypersurface in the first factor or are explicitly parametrized by using the parametrization of a semi-parallel hypersurface in the first factor with a linear parametrization in the second factor in the way given in \cite{To}.  A classification of the pseudo-parallel hypersurfaces of $S^m\times \r$ and $H^m\times \r$ which are minimal or have constant mean curvature is given in  \cite{LT}.

A hypersurface in $S^m\times\mathbb R$ (resp. $H^m\times\mathbb R$) is said to be {\it normally flat} if it has flat normal bundle when viewed as a codimensional 2 submanifold in $\r^{m+2}\supset S^m\times\mathbb R$ (resp. $\mathbb{L}^{m+2}\supset H^m\times\mathbb R$). It was proved in \cite{DFV09} and \cite{DMN} for the case of $m=2$, and in \cite{To} for the general case that a hypersurface in $S^m\times\mathbb R$ or $H^m\times\mathbb R$ is normally flat if and only if  $T$, the tangent component of $\partial_t$ is a principal direction. The results of \cite{DFV09}, \cite{DMN}, and \cite{To} also show that the family of normally flat hypersurfaces includes both the families of rotation hypersurfaces and that of constant angle hypersurfaces as proper subsets.

In this paper, we study biharmonic hypersurfaces in a
product $L^m\times \mathbb{R}$ of an Einstein space $L^m$ and a real
line $\mathbb R$. Recall that a hypersurface is {\it biharmonic} if the isometric immersion defining the hypersurface is a biharmonic map. For a recent survey on the study of biharmonic submanifolds see \cite{Ou3}. It was proved in \cite{Ou1} that a hypersurface
$\varphi:M^{m}\to N^{m+1}$  with mean curvature vector
$\eta=H\xi$ is biharmonic if and only if
\begin{equation}\label{BHEq}
\begin{cases}
\Delta H-H |A|^{2}+H{\rm
Ric}^N(\xi,\xi)=0,\\
 2A\,(\nabla H) +\frac{m}{2} \nabla H^2
-2\, H \,({\rm Ric}^N\,(\xi))^{\top}=0,
\end{cases}
\end{equation}
where ${\rm Ric}^N : T_qN\to T_qN$ denotes the Ricci
operator of the ambient space defined by $\langle {\rm Ric}^N\, (Z),
W\rangle={\rm Ric}^N (Z, W)$.

For the study of biharmonic hypersurfaces in $S^m\times \r$ and $H^m\times \r$, it was proved in \cite{OW} that the only proper biharmonic surface with constant mean curvature in $S^2\times \r$ and $H^2\times \r$ is an open subset of the vertical cylinder $S^1(\frac{1}{\sqrt{2}})\times \r$, and that there is no totally umbilical proper biharmonic surface in $S^2\times \r$ and $H^2\times \r$. The result on the case of $S^2\times \r$ was later generalized \cite{FOR} to the case of $S^m\times \r$. Note that for the higher dimension, even we know that a proper biharmonic hypersurface in $S^m\times \r$ is a vertical cylinder $M^{m-1}\times S^m$ where $M^{m-1}$ is a proper biharmonic hypersurface of the sphere $S^m$, the complete picture is still missing as the classification of biharmonic hypersurfaces of a sphere is still far from our reach. For more study of proper biharmonic submanifolds with parallel mean curvature vector field  in $S^m\times \r$ see \cite{FOR}, and for some classification of biconservative surfaces (a class of surfaces that contains biharmonic surfaces as a subclass) with parallel mean curvature vector field in $S^m\times \mathbb{R}$ and $H^m\times \mathbb{R}$ see \cite{FOP}. 

The rest of the paper is organized as follows. In Section \ref{SSC}, we compute the Laplacian of the mean curvature function of a biharmonic hypersurface in the product $M^m\times \r$, and apply it to prove, among other things,  that a
biharmonic hypersurface with constant mean curvature in a product $L^m\times \mathbb{R}$ of an
Einstein space and a line is either minimal or a vertical cylinder (Theorem \ref{ConstH}). In Section 3, we first derive the  biharmonic equation for hypersurfaces in $S^m\times \mathbb{R}$ and $H^m\times \mathbb{R}$ in terms of the angle functions of the hyperesurfaces (Lemma \ref{ML2}). Then, as applications, we use the biharmonic equation to obtain a
complete classification of constant angle biharmonic surfaces in
$S^2\times \mathbb{R}$ and $H^2\times \mathbb{R}$ (Theorem
\ref{MT}), and to have a system of ordinary differential equations for rotation biharmonic hypersurfaces in $S^m\times \mathbb{R}$ (Theorem
\ref{RE1}). Utilizing these  equations we give a classification of biharmonic hypersurfaces in $S^m\times \mathbb{R}$ and $H^m\times \mathbb{R}$ which are totally umbilical or semi-parallel for $m\ge 3$ (Theorem \ref{semibih}) in Section 4. In Section 5, by using the parametrizations of rotation surfaces introduced
in \cite{AEG}, we obtain some classification results on biharmonic
rotation surfaces in $S^2\times\r$ and $H^2\times\r$.

Throughout the paper, we assume that a hypersurface
$\varphi:M\rightarrow (L\times\mathbb{R},g^L+dt^2)$ is two-sided, which means that there exists a globally defined unit normal vector
field.
\section{biharmonic hypersurfaces in a product of Einstein spaces}\label{SSC}
First, we recall the following corollary which will be used in several places in the paper.
\begin{corollary}\cite{Ou2}\label{O10}
A hypersurface $\varphi: M^m\to P^m\times\mathbb{R}$ in the product space is biharmonic if and only if  both its component maps $\pi_1\circ\varphi: (M, g)\to (P, g^P)$  and $\pi_2\circ\varphi: (M, g)\to (\r, {\rm d}t^2)$ are biharmonic maps with respect to the induced metric $g=\varphi^{*}(g^P+{\rm d}t^2)$. In particular, the height function $h=\pi_2\circ\varphi$ of a biharmonic hypersurface is a biharmonic function on the hypersurface.
\end{corollary}
An immediate consequence of this and the maximum principle for Laplace operator is the following
\corollary\label{O11}(see also \cite{FOR})
There is no compact proper biharmonic hypersurface in the product manifold $ P^m\times\mathbb{R}$ for any Riemannian manifold $(P^m, g)$.
\endcorollary
The angle function $\theta=\langle \xi,\frac{\partial}{\partial t}
\rangle$ for a hypersurface  $M^m\to (P^m\times\mathbb{R}, g^P+dt^2)$, where $\xi$ is the unit normal vector field of the hypersurface, has played an important role in the study of the geometry of the hypersurfaces in the product space. The following Laplacian of the angle function
$\theta$, computed in (cf. \cite{AAA}), will also be used in our paper.
\begin{align}\label{la}
\Delta\theta=-m\langle\nabla H,\partial_t\rangle-\theta(|A|^2+{\rm
Ric}^N(\xi,\xi)),
\end{align}
where $\partial_t=\frac{\partial}{\partial t}$.
 Now, we prove the following lemma which gives the Laplacian of the mean curvature of the hypersurface in terms of the angle function.
\begin{lemma}\label{ML}
Let $\varphi:M^m\rightarrow (P^m\times\mathbb{R}, g^P+dt^2)$ be a biharmonic  hypersurface in $P\times\mathbb{R}$.
Then, we have the following identity
\begin{align}\label{HT}
0=\Delta(H\theta)=\theta\Delta H+2\langle\nabla H,\nabla \theta\rangle+H\Delta\theta.
\end{align}
\end{lemma}
\begin{proof}
Let  $h=\pi_2\circ \varphi$,  as in Corollary \ref{O10}, be the height function of the hypersurface, then, one can check (see also \cite{AAA}) that
$\Delta h=m\theta H$. A further computation yields
\begin{align}\notag
\Delta^2h=\Delta(H\theta)=\Delta H\theta+2\langle\nabla H,\nabla \theta\rangle+H\Delta\theta,
\end{align}
from which, together with the last statement of Corollary \ref{O10},
we obtain the lemma.
\end{proof}
Now, we are ready to prove the following theorem.
\begin{theorem}\label{ConstH}
Let $(L^m, g^L)$ be an Einstein manifold with ${\rm Ric}^L=\lambda g^L$. Then, a  constant mean curvature biharmonic hypersurface $\varphi:M^m\to (L^m\times\mathbb{R}, g^L+dt^2)$ is either minimal, or a vertical cylinder over a biharmonic hypersurface in $L^m$, i.e.,
$\varphi(M)=\phi(M^{m-1})\times \mathbb{R}$, where $\phi:
M^{m-1}\to (L^m, g^L)$ is a biharmonic hypersurface in $(L^m, g^L)$.
\end{theorem}
\begin{proof}
If $H=0$, then $M$ is minimal. Now, if the constant $H\not=0$, then, by $(\ref{HT})$, we have $H\Delta\theta=0$ which implies that $\Delta\theta=0$. On the other hand, one can check that ${\rm Ric}^N(\xi,\xi)=\lambda(1-\theta^2)$ from which, together with $(\ref{la})$, we have
\begin{equation}\label{eq1}
\theta(|A|^2+\lambda(1-\theta^2))=0.
\end{equation}
 Now, using the  first equation of $(\ref{BHEq})$, we have
\begin{equation}\label{eq2}
|A|^2=\lambda(1-\theta^2).
\end{equation}
Combining (\ref{eq1}) and (\ref{eq2}) we have $2\theta\,|A|^2=0$. It follows that  $\theta\equiv0$ since otherwise, there would be a neighborhood on which $|A|^2=0$ and hence $H\equiv 0$, which is a contradiction.

Noting that $\theta=\langle \xi, \partial_t\rangle\equiv 0$ means exactly that $\partial_t$ is tangent to the hypersurface, we conclude that the hypersurface $M$ is a vertical cylinder, i.e., $\varphi(M)=\phi(M^{m-1})\times \mathbb{R}$, where $\phi: M^{m-1}\to L^m$ is a biharmonic (by Corollary \ref{O10}) hypersurface  in the Einstein space $(L^m, g^L)$.
\end{proof}
\begin{remark}
Note that if the Einstein space $L$ is a sphere, then Theorem \ref{ConstH} recovers a part of a result in \cite{FOR}, which gives more specific descriptions of biharmonic vertical cylinders in $S^m(r)\times \r$:

\noindent $(i)$ If $m=2$, then $M^1$ is a circle in $S^2$ with
curvature equal to 1 and $|\eta|=\frac{1}{2}$;

\noindent $(ii)$ If $m=3$, $M^2$ is an open part of a small
hypersphere $S^2(2)\subset S^3$ and $|\eta|=\frac{2}{3}$;

\noindent $(iii)$ if $m>3$, then $|\eta|\in
(0,\frac{m-3}{m}]\cup\{\frac{m-1}{m}\}$; Furthermore,

$(a)$ $|\eta|=\frac{m-3}{m}$ if and only if $M^{m-1}$ is an open
part of the standard product $S^{m-2}(2)\times S^1(2)\subset S^m,$

$(b)$ $|\eta|=\frac{m-1}{m}$ if and only if $M^{m-1}$ is an open
part of a small hypersphere $S^{m-1}(2)\subset S^m,$

where $\eta$ is the mean curvature vector field of $M$.
\end{remark}

Parallel to the compact case in Corollary \ref{O11},  we can use Yau's Maximum principle to have the following.
%%%%%%%%%%%%

\begin{corollary}\label{cor2}
Let $(L^m, g^L)$ is an Einstein manifold and $\varphi:M^m\rightarrow (L^m\times\mathbb{R},g^L+dt^2)$ be a complete constant angle biharmonic hypersurface.

$(i)$ If the mean curvature $H$ is positive and $M$ has nonnegative Ricci curvature, then $M$ is a vertical cylinder over a biharmonic hypersurface in the Einstein space $(L^m, g^L)$.

$(ii)$ If the mean curvature $H$ is nonnegative and $H\in L^p(M)$ for $1<p<\infty$, then $M$ is minimal, or a vertical cylinder over a biharmonic hypersurface in $(L^m, g^L)$.
 \end{corollary}
\begin{proof}
If $\theta$ is a nonzero constant, then, by $(\ref{HT})$,  we have $\Delta H=0$, from which, together with Yau's maximum principle,
we have $H$ is constant. The corollary then follows from Theorem $\ref{ConstH}$.
\end{proof}
\begin{corollary}\label{umb}
Let $(L^m, g^L)$ is an Einstein manifold with ${\rm Ric}^L=\lambda g^L$ and $\varphi:M^m\rightarrow (L^m\times\mathbb{R}, g^L+dt^2)$ be a totally umbilical biharmonic hypersurface with constant angle function, then it is either minimal, or a vertical cylinder over a biharmonic hypersurface in the Einstein space $(L^m, g^L)$.
\end{corollary}
\begin{proof}
Note that it was proved in \cite{BMO3} (see also \cite{DO}) that any totally umbilical biharmonic submanifold $M^m$ with $m\not=4$ has constant mean curvature. Using this, together with Theorem \ref{ConstH}, we obtain the corollary for the case of $m\ne 4$. Now for $m=4$, since $M^4$ is totally umbilical, we can choose an an orthonormal frame $\{e_1,\cdots, e_4\}$ so that  $A(e_i)=H e_i$ for $i=1, 2, 3, 4$. It follows that $|A|^2=4H^2$.
Using the first equation in (\ref{BHEq}) we have
$$\Delta H -4H^3+\lambda H(1-\theta^2)=0.$$
If $\theta\not=0$, by $(\ref{HT})$, $\Delta H=0$.
Therefore we have
$$-4H^3+\lambda H(1-\theta^2)=0,$$
which means that $H$ is constant.
Thus, Theorem \ref{ConstH} applies to complete the proof.
\end{proof}

To prove the proposition we will use a well known Yau's Maximum principle:

\begin{theorem}
$(a)$ Let $u$ be a non-negative smooth subharmonic function on a complete Riemannian manifold $M$.
 Then $\int_Mu^p=+\infty$ for $p>1$, unless $u$ is a constant function.

$(b)$ Let $u$ be a positive smooth harmonic function on a complete Riemannian manifold with non-negative Ricci curvature. Then $u$ is a constant function.
\end{theorem}

We will use the following Liouville type theorem:

\begin{theorem}[\cite{LM}]\label{ThmLM}
Let $(M,g)$ be a complete noncompact manifold and $u\in(0,C],~(C>0)$ a superharmonic function on $M$. If
\begin{eqnarray}\notag
\int_M(\log^{(k)}\frac{Ce^{(k)}}{u})^pdv_g<+\infty
\end{eqnarray}
for some $p>0$ and $k\in \mathbb{N}$, then $u$ is a constant.
Here $\log^{(k)}=\log(\log^{(k-1)})$ and $e^{(k)}=e^{e^{(k-1)}}$, where $\log^{(1)}=\log$ and $e^{(1)}=e$.
\end{theorem}
\begin{proposition}\label{main2}
Let $L$ is an Einstein manifold and $\varphi:M^m\rightarrow (L^m\times\mathbb{R},g^L+dt^2)$ be a complete biharmonic hypersurface with non-negative Ricci curvature.
Assume that
$$
\int_M H^p dv_g<+\infty,\quad\text{for some}~p>2,
$$
and
\begin{eqnarray*}
\int_M(\log^{(k)}\frac{e^{(k)}}{\theta^2+\varepsilon})^qdv_g<+\infty,
\end{eqnarray*}
for some $q>0$, $k\in \mathbb{N}$ and $\epsilon>0$.
Then, $M$ is minimal, or a vertical cylinder over a biharmonic hypersurface in $(L^m, g^L)$.
\end{proposition}

\begin{proof}%[Proof of Proposition \ref{main2}]
By Lemma \ref{ML}, we have
$$\Delta (H\theta)^2=2|\nabla (H\theta)|^2\geq0.$$
Since $-1\leq\theta\leq1,$
$$\int_M\{(H\theta)^2\}^{q}dv_g\leq\int_MH^{2q}dv_g<+\infty,$$
for $q>1$.
By Yau's Maximum principle, we have $H\theta$ is constant $\bar C$.

By the Ricci identity,
$$\Delta \nabla_i h=\nabla_i\Delta h+{\rm Ric}^M_{ij}\nabla_jh={\rm Ric}^M_{ij}\nabla_jh,$$
where the second equality follows from $\Delta h=mH\theta=m\bar C$. So we have,
$$\langle \Delta T,T\rangle ={\rm Ric}^M(T,T).$$
Hence,
\begin{align*}
-\frac{1}{2}\Delta(\theta^2+\varepsilon)=\frac{1}{2}\Delta|T|^2
=|\nabla T|^2+\langle \Delta T, T\rangle
=|\nabla T|^2+{\rm Ric}^M(T,T)\geq0.
\end{align*}
Therefore $\Delta (\theta^2+\varepsilon)\leq0$, that is, $\theta^2+\varepsilon(>0)$ is a superharmonic function on $M$.
By Theorem \ref{ThmLM}, we obtain $\theta^2+\varepsilon$ is constant. Hence, $\theta$ is constant.
From this and $H\theta=\bar C$, $H$ is constant.
By Theorem \ref{ConstH}, the proof is complete.
\end{proof}
\begin{proposition}\label{main3}
Let $(L^m, g^L)$ is an Einstein manifold with ${\rm Ric}^L=\lambda g^L$ and $\varphi:M^m\rightarrow (L^m\times\mathbb{R}, g^L+dt^2)$ be a complete biharmonic hypersurface with non-negative Ricci curvature.  Assume that

$(i)$ $H$ is harmonic and bounded from below, or

$(ii)$ $\theta$ is harmonic and the scalar curvature of $M$ is constant.

Then $M$ is minimal, or a vertical cylinder over a biharmonic hypersurface in  $(L^m, g^L)$.
\end{proposition}
\begin{proof}
(i) Since $H$ is bounded from below by some constant $-C$, $u=H+C+\varepsilon$ is positive. Since $\Delta u=\Delta H=0$,  by Yau's maximum principle, $u$ is constant. Hence, $H$ is constant. By Theorem \ref{ConstH}, the proof is complete.

(ii) Since $-1\leq\theta\leq1$, $u=\theta+2$ is positive. Since $\Delta u=\Delta \theta=0$, by Yau's maximum principle, $u$ is constant. Hence, $\theta$ is constant. Assume that $\theta\not=0$. By Lemma \ref{ML}, $\Delta H=0.$ By the first equation of (\ref{BHEq}),
$H(|A|^2-\lambda(1-\theta^2))=0.$
Assume that $H\not=0$ at $p\in M$, that is, $H\not=0$ on some
neighborhood $\Omega\ni p$. Then $|A|^2=\lambda(1-\theta^2).$ By
Gauss equation and the relationships between the Ricci curvatures
and scalar curvatures of the hypersurface and the ambient space,
respectively, one gets
$\lambda m={\rm Scal}^M-m^2H^2+3\lambda(1-\theta^2),$
where we used ${\rm Scal}^N=\lambda m.$
So we have
$\nabla\, {\rm Scal}^M=2m^2H\nabla H.$
Since the scalar curvature of $M$ is constant, $H$ is constant. By
Theorem \ref{ConstH}, the proof is complete.
\end{proof}

\section{Biharmonic hypersurfaces in $L^m(c)\times \mathbb R$}

We consider a biharmonic hypersurface $M^m$ in $L^m(c)\times \mathbb
R$, where $L^m(c)$ is the space form $S^m$, $H^m$ or $E^m$ with
constant curvature $c=1$, $-1$ or $0$, respectively.

The Riemannian curvature tensor of $L(c)\times\mathbb{R}$ is given
by
\begin{align}\label{eq:tildeR}
R^N(X,Y)Z=c\{\langle Y,Z\rangle X-\langle X,Z\rangle Y -\langle
Y,\partial_t\rangle\langle Z,\partial_t\rangle X
+\langle X,\partial_t\rangle\langle Z,\partial_t\rangle Y\nonumber\\
+\langle X,Z\rangle\langle Y,\partial_t\rangle \partial_t -\langle
Y,Z\rangle\langle X,\partial_t\rangle \partial_t\},
\end{align}
where $X, Y, Z$ are vector fields on $L(c)\times\mathbb{R}$.

Since $\partial_t$ is a unit vector field globally defined on the
ambient space $L^m(c)\times \mathbb R$, we can decompose it in the
following form
\begin{equation}
\label{eq:dt}
\partial_t=T+\cos\alpha\xi,
\end{equation}
where $\cos\alpha= \langle \partial_t, \xi\rangle$  with $\alpha$ denoting the angle made by $\partial_t$ and the unit normal vector field of the hypersurface, and  $T$ denotes the tangential component of $\partial_t$ along the tangent plane to
$M^m$. Note that here $\cos\alpha=\theta$ related to the notation used in the previous
sections.

For any vector fields $U$, $V$, $W$ tangent to $L^m(c)\times \mathbb
R$, the Codazzi equation is given by
\begin{eqnarray}
\label{eq:Codazzieq} (\nabla_XA)Y-(\nabla_YA)X = c
\cos\alpha(\langle Y,T\rangle X-\langle X,T\rangle Y),
\end{eqnarray}
where $X$ and $Y$ are tangent vector fields on $M^m$.

Since $\partial_t$ is parallel on $L^m(c)\times \mathbb R$, a direct
computation yields
\begin{eqnarray}
\nabla_XT&=&\cos\alpha AX,\label{eq:nablaXT}\\
X(\cos\alpha)&=&-\langle AX,T\rangle, \label{eq:XTheta}
\end{eqnarray}
for every tangent vector field $X$ on $M^m$.

In terms of angle function $\alpha$,  the
biharmonic equations \eqref{BHEq} can be rewritten in the
following form.
\begin{lemma}\label{ML2} A hypersurface  $\varphi:M^{m}\to L^m(c)\times \mathbb{R}$  with mean curvature vector $\eta=H\xi$ is biharmonic if and only if
\begin{equation}\label{BHP1}
\begin{cases}
\Delta H-H [|A|^{2}-c(m-1)\sin^2\alpha ]=0,\\
 A\,(\nabla H) +\frac{m}{2}H \nabla H
+c(m-1)\cos \alpha H T=0,
\end{cases}
\end{equation}
where $T$ is the tangential component of $\partial_t$ and $\cos
\alpha=\theta=\langle \partial_t, \xi\rangle$.
\end{lemma}
\begin{proof}
Choose a local orthonormal frame $\{e_i\}, i=1,\ldots,m$ on $M^m$. Then, a straightforward computation using \eqref{eq:tildeR} and \eqref{eq:dt} yields
\begin{eqnarray*}
{\rm Ric}^N(\xi,\xi)=\sum_{i=1}^{m}{\rm R}^N(\xi, e_i,\xi,
e_i)=\sum_{i=1}^{m}\langle{\rm R}^N( e_i, \xi)\xi,
e_i\rangle=c(m-1)\sin^2\alpha,
\end{eqnarray*}
and
\begin{eqnarray*}
({\rm Ric}^N\,(\xi))^{\top}=\sum_{i=1}^{m}\langle{\rm R}^N( e_i,
\xi)e_k, e_i\rangle e_k=-c(m-1)\cos\alpha T.
\end{eqnarray*}
Substituting these into the biharmonic equations
\eqref{BHEq}, we obtain the lemma.
\end{proof}
\begin{remark}
We remark that Equation (\ref{BHP1}) generalizes the  biharmonic
equations for hypersurfaces in a Euclidean space, which is
useful in the study any biharmonic hypersurfaces in $L^m(c)\times
\mathbb{R}$, including rotation hypersurfaces, semi-parallel or more
general ones.
\end{remark}
Now we are ready to give a complete classification of constant angle biharmonic surfaces in $S^2\times\mathbb R$ and
$H^2\times\mathbb R$.
\begin{theorem}\label{MT}
The only constant angle proper biharmonic surface in
$S^2\times\mathbb R$ and $H^2\times\mathbb R$ is an open part of
$S^1(1/\sqrt2)\times\mathbb R$.
\end{theorem}
\begin{proof}
 Choose a suitable frame $\{e_1,e_2,\xi\}$ so that the shape operator is diagonalized
with $Ae_1=\lambda_1e_1$ and $Ae_2=\lambda_2e_2$.  The fact $\langle
T,T\rangle=\sin^2\alpha$ implies that
\begin{eqnarray}\label{eq:T}
T=\sin\alpha(\cos fe_1+\sin f e_2)
\end{eqnarray}
for some smooth function $f$ on $M$. With this orthonormal frame
$\{e_1,e_2,\xi\}$ and  \eqref{eq:XTheta}, we have
\begin{eqnarray}\label{eq:theta}
e_1(\alpha)=\lambda_1\cos f,\quad e_2(\alpha)=\lambda_2\sin f.
\end{eqnarray}
Setting
\begin{eqnarray}\label{eq:nabla}
\nabla_{X}e_2=\omega(X)e_1,\quad\nabla_{X}e_1=-\omega(X)e_2,
\end{eqnarray}
 and substituting \eqref{eq:T} into \eqref{eq:nablaXT}, we find
\begin{eqnarray}\label{eq:omega}
\omega(e_1)=e_1(f)+\lambda_1\cot\alpha\sin f,\quad
\omega(e_2)=e_2(f)-\lambda_2\cot\alpha\cos f.
\end{eqnarray}
 Thus, the biharmonic equations
\eqref{BHP1} becomes
\begin{equation}\label{BhE3}
\begin{cases}
\Delta H-H |A|^{2}+c\sin^2\alpha H=0,\\
 A\,({\rm grad}\,H) +H {\rm grad}\, H
+c\cos\alpha\sin\alpha\, H \,(\cos f e_1+\sin f e_2)=0.
\end{cases}
\end{equation}
On the other hand,  from Codazzi equation we have
\begin{equation}\label{Codazzi}
\begin{cases}
e_1(\lambda_2)=(\lambda_2-\lambda_1)\omega(e_2)-c\sin\alpha\cos\alpha\cos f,\\
e_2(\lambda_1)=(\lambda_2-\lambda_1)\omega(e_1)-c\sin\alpha\cos\alpha\sin f.
\end{cases}
\end{equation}
Also, from the second equation of (\ref{BhE3}) we have
\begin{equation}\label{Bih2}
\begin{cases}
(\lambda_1+H)e_1(H)+c\sin\alpha\cos\alpha\cos fH=0,\\
(\lambda_2+H)e_2(H)+c\sin\alpha\cos\alpha\sin fH=0.
\end{cases}
\end{equation}
Since the surface is proper biharmonic, it is not minimal.  Using this, together with the assumption that the
angle function $\alpha$ is always constant and
\eqref{eq:theta}, we conclude that either $\lambda_1=\sin f=0$ or $\lambda_2=\cos
f=0$. For the first case, we use \eqref{eq:omega}, \eqref{Codazzi} and \eqref{Bih2} to have
\begin{eqnarray*}
&&e_1(\lambda_2)=-\lambda_2^2\cot\alpha-c\sin\alpha\cos\alpha,\\
&&e_1(\lambda_2)=-2c\sin\alpha\cos\alpha,
\end{eqnarray*}
which yield that $\lambda_2=c\sin^2\alpha$ and hence $\lambda_2$ is
a constant. Consequently, the mean curvature $H=\lambda_2/2$ is constant. Similarly, one can check that
the second case also leads to  constant mean curvature
$H$. So, in either case, we can use the classification of constant mean curvature biharmonic surfaces in $S^2\times\mathbb R$ and $H^2\times\mathbb R$ given in \cite{OW} to conclude.
\end{proof}
In the following, we will study the biharmonicity of the rotation hypersurfaces $M^m$
in $S^m\times \mathbb R$ defined by Dillen et al. (c.f. \cite{DFV}).
Parametrizing the profile curve as
$$\gamma(s)=(\cos s, 0,\ldots, 0, \sin s, h(s)),$$
for some smooth function $a$,  then the parametrization of the
rotation hypersurface can be written as
\begin{align}\notag
& f(s, v_1, \ldots, v_{m-1})\\\label{Rotm}&=\big(\cos s, \varphi_1(v_1, \ldots, v_{m-1})\sin s, \ldots,
\varphi_n(v_1, \ldots, v_{m-1})\sin s, h(s)\big),
\end{align}
where $\varphi=(\varphi_1, \ldots, \varphi_m)$ is an orthogonal parametrization of the unit sphere $S^{m-1}$ in $\mathbb R^m$. 

Our result can be stated as
\begin{theorem}\label{RE1}
A rotation hypersurface in $S^m\times \r$ defined in (\ref{Rotm}) is biharmonic if  its mean curvature function $H$ solves the equations
\begin{eqnarray}\label{54}
&&\Big(3u'+(m-1)u\cot s\Big)H'+2(m-1)u H=0,
\\\label{55} &&(1-u^2)H''+\Big((m-1)(1-u^2)\cot s-uu'\Big)H'
\\&&+\Big((m-1)u^2(1-\cot^2s)-u'^2\Big)H=0.\nonumber
\end{eqnarray}
where $u=-\sin\alpha$ and 
$H=\frac{1}{m}\big(u'+(m-1)u\cot s\big)$.
\end{theorem}
\begin{proof}
As in \cite{DFV}, we  choose an orthonormal frame
\begin{equation}\notag
e_1=\frac{1}{\sqrt{1+h'(s)^2}}df(\partial_
s),\;\;e_i=\frac{1}{\sqrt{\sin^2
s(\sum_{k=1}^{m}\frac{\partial \varphi_k}{\partial
v_i})}}df(\partial_i),\quad 2\leq i\leq m.
\end{equation}
on the rotation hypersurface with the unit normal vector field\\ $\xi=\frac{1}{\sqrt{1+h'(s)^2}}(-h'(s)\sin s, \varphi_1h'(s)\cos s ,
\ldots, \varphi_m h'(s)\cos s, -1)$ so that 
\begin{eqnarray}\notag
&& A(e_1)=\lambda e_1, \;\;\lambda=-\frac{h''(s)}{(1+h'(s)^2)^{3/2}},\label{lambda}\\\notag
&&A(e_j)=\mu e_j,\;\;\;\mu=-\frac{h'(s)\cot s}{(1+h'(s)^2)^{1/2}}, \quad j=2.
\ldots, m,\label{mu}
\end{eqnarray}
It is easy to check that
\begin{align}\notag
\cos\alpha(s) =&\langle \xi, \partial_t\rangle=-\frac{1}{\sqrt{1+h'(s)^2}},\;\;\sin\alpha(s) =\frac{h'(s)}{\sqrt{1+h'(s)^2}},\\\notag
\lambda=&-\alpha'(s)\cos\alpha,\quad \mu=-\sin\alpha\cot s,\label{alpha6}\;\; {\rm and}\\\notag
 H=&-\frac{1}{m}\big(\alpha'(s)\cos\alpha+(m-1)\sin\alpha\cot s\big).
\end{align}
A further computation using the fact that  $e_i(H)=e_i(\lambda)=e_i(\mu)=0$ for $2\leq i\leq m$, and that
$\nabla H=\sum_{i=1}^m e_i(H)e_i$, we can rewrite the second
equation of \eqref{BHP1} as
\begin{eqnarray}\label{T}
e_1(H)\lambda e_1+\frac{m}{2}He_1(H)e_1+(m-1)\cos\alpha HT=0.
\end{eqnarray}
It follows that if $H\ne 0$ in an open set, then $T$ is proportional to $e_1$, which, together with the fact that
$\langle T, T\rangle=\sin^2\alpha$, allows us to write 
\begin{equation}\label{T-sin}
T=\sin\alpha e_1.
\end{equation}
So, (\ref{T}) becomes 
\begin{eqnarray}\label{eq:Main1}
(\frac{m}{2}H+\lambda)e_1(H)=-(m-1)\sin\alpha\cos\alpha H.
\end{eqnarray}
To compute the the term $\Delta H$,  we first use (\ref{eq:nablaXT}) to compute
\begin{align}\label{521}
 \mu\cos \alpha =&\langle \nabla_{e_j} T, e_j\rangle =e_j\langle  T, e_j\rangle -\langle T, \nabla_{e_j} e_j\rangle \\\notag
=&-\sin \alpha \langle e_1, \nabla_{e_j} e_j\rangle.
\end{align}
It follows that
\begin{align}\label{522}
\langle e_1, \nabla_{e_j} e_j\rangle=- \mu\cot \alpha,
\end{align}
from which we have
\begin{eqnarray}\notag
&&\Delta H=\sum_{i=1}^m \big(e_ie_i(H)-\langle\nabla_{e_i}e_i, e_1\rangle e_1(H)\big)\\\label{523}
&&=e_1e_1(H)+(m-1)\cot\alpha\mu e_1(H),
\end{eqnarray}
By using this and  $|A|^{2}=\lambda^2+(m-1)\mu^2$, we can rewrite the first equation of \eqref{BHP1} as
\begin{eqnarray}\label{51}
e_1e_1(H)+(m-1)\cot\alpha\mu
e_1(H)-H(\lambda^2+(m-1)\mu^2)+\sin^2\alpha H=0.
\end{eqnarray}
Finally, by a change of variable 
\begin{equation}\label{change}
u=-\sin \alpha(s),
\end{equation}
 and using
\begin{eqnarray}\label{53}
&&\lambda=u',\quad \mu=u\cot s,\nonumber\\
&&H=\frac{1}{m}\big(u'+(m-1)u\cot s\big), \nonumber\\
&&e_1(H)=-\cos\alpha H',\nonumber
\end{eqnarray}
in Equations (\ref{eq:Main1}) and  (\ref{51}), we obtain the two equations stated in the theorem.
\end{proof}
%%%%%%%%%%%%%%%%%%%%%%%%%%%%%%%%%

\section{Semi-parallel biharmonic hypersurfaces in $L^m(c)\times \mathbb R$}

In this section, we first give a complete classification for totally umbilical biharmonic hypersuraces in $L^m(c)\times\mathbb{R}$, where
$L^m(1)=S^m$ and $L^m(-1)=H^m$, then we use the results to classify semi-parallel biharmonic hypersurfaces in such spaces. For the existence of general  totally umbilical hypersurfaces in $L^m\times \mathbb{R}$ see \cite{SV}, and for the study of totally umbilical hypersurfaces in  $L^m(c)\times\mathbb{R}$ see \cite{CKV}, \cite{VV} and \cite{MT14}. 
\begin{theorem}\label{umbbih}
Any totally umbilical biharmonic hypersurface in $L^m(c)\times\mathbb{R}$ is minimal.
\end{theorem}
\begin{proof}
As we have seen in Theorem \ref{ConstH} (also \cite{FOR}) a constant mean curvature hypersurface in $L^m(c)\times\mathbb{R}$ is biharmonic if and only if it is minimal or a vertical cylinder over a biharmonic hypersurface. Since a vertical cylinder is not totally umbilical, it is enough to show that a totally umbilical biharmonic hypersurface in $L^m(c)\times\mathbb{R}$ has constant mean curvature.  Since a totally umbilical biharmonic hypersurface of dimension $m\ne 4$ always has constant mean curvature, we only need to do the proof for the case of $m=4$. In this case, the two equations of \eqref{BHP1} read
\begin{equation}\label{BHEq2}
 \left\{
 \begin{aligned}
 &\nabla H +c\cos\alpha T=0,\\
 &\Delta H -4H^3+3c\sin^2\alpha H=0.
 \end{aligned}
\right.
\end{equation}
If $H\equiv0$, then the proof completes. Otherwise, we assume that $H\not=0$ on an open
set $\Omega$, and we will consider the equations on $\Omega$.

If $\sin\alpha\equiv0$, then $\cos\alpha=\pm1$. By \eqref{eq:XTheta} and the first equation of \eqref{BHEq2},
we have $H|\nabla H|^2=0$, which implies that $|\nabla H|^2=0$ and hence $H$ is non-zero constant by the assumption  that $H\ne 0$ on $\Omega$.

Now if $\sin\alpha\not=0$ at a point $p\in M^4$. Then, it is
shown (cf.\cite{VV} and \cite{CKV}) that there exists local coordinates $(u,v_1, v_2, v_{3})$
on an open neighborhood of $p$ such that
\begin{align}\notag
\partial_u=\frac{1}{\sin\alpha}T,\qquad\partial_u\perp\partial_{v_i},
\qquad \partial_{v_i}H=\partial_{v_i}\alpha=0,
\end{align}
and $\phi:=2\alpha$ solving the Sine-Gordon equation
\begin{equation}\label{SG}
\phi''+c\sin\phi=0.
\end{equation}
First, we note, by using \eqref{eq:XTheta}, that $\alpha'=H$. Since
$\nabla H=\partial_u H\,\partial_u$, the first equation of
\eqref{BHEq2} reads
\begin{equation}\label{BHEq2T}
H'=-c\cos\alpha\sin\alpha.
\end{equation}
Second, recalling that $|T|^2=\sin \alpha$ we see that the existence of the local coordinates $(u,v_1, v_2, v_{3})$ implies the existence of a local orthonormal frame 
$\{e_1, e_2, e_3, e_4\}$ on $M^4$ such that $e_1=\partial_u$ and hence $T=\sin \alpha e_1$,  $e_j(\alpha)=0$ for $j\ge 2$, and that
\begin{equation}\notag
A(e_i)=H e_1, \;\; i=1, 2, 3, 4.
\end{equation}
Now, exactly as in the calculations of (\ref{521}), (\ref{522}), and (\ref{523}), we have
\begin{align*}
\Delta H=&H''+3\cot\alpha HH'.
\end{align*}
Therefore, the second equation of \eqref{BHEq2} becomes
\begin{align}\label{BHEq2N}
H''+3\cot\alpha H H'-4H^3+3c\sin^2\alpha H=0.
\end{align}
Differentiating \eqref{BHEq2T} and combining the resulting equation with 
\eqref{BHEq2N} yields
\begin{align}\label{524}
-\alpha'(4c\cos(2\alpha)+(2\alpha')^2)=0.
\end{align}
If $\alpha'\equiv0$, then $H=\alpha'\equiv0$ and we have a contradiction.  If otherwise, we consider equations on some neighborhood $\Omega$ on which $\alpha'\not=0.$
Denoting $\phi:=2\alpha$, Equation (\ref{524}) reads
$$(\phi')^2+4c\cos\phi=0.$$
By differentiating this and combining the resulting equation with \eqref{SG}, we obtain
$(\cos\phi)'=0$, and hence $\alpha'=0$, which is a contradiction.

Summarizing the above discussion we obtain the conclusion about the case of ${S}^m\times\mathbb{R}$.

Finally, we can check that  an argument similar to the above works for the case of ${H}^m\times\mathbb{R}$.
\end{proof}
\begin{remark}
Note that Theorem \ref{umbbih} implies that there is no totally umbilical proper biharmonic hypersurface in the conformally flat space $L^m(c)\times\mathbb{R}$. However, it was proved in \cite{OT} that there are many totally umbilical proper biharmonic hypersurfaces (with constant mean curvature) in other conformally flat spaces.  Also, we would like to point out that Theorem \ref{umbbih} holds for the conformally flat space $L^m(c)\times\mathbb{R}$, but it cannot be generalized to a general conformally flat space. This is evident by Example 1 in \cite{LO} where many examples of totally umbilical proper biharmonic hypersurfaces of dimension $4$ with non-constant mean curvature are constructed in a conformally flat space.
\end{remark}
Now we are ready to give a classification of  semi-parallel biharmonic hypersurfacers in $L^m(c)\times\mathbb{R}$.
\begin{theorem}\label{semibih}
(i) Any semi-parallel biharmonic  hypersurface in $S^m\times\mathbb{R}$ for $m\ge 3$ is minimal or a vertical cylinder over a biharmonic hypersurface in 
$S^m$;\\
(ii) Any semi-parallel biharmonic  hypersurface in $H^m\times\mathbb{R}$ for $m\ge 3$ is minimal.
\end{theorem}
\begin{proof}
For Statement (i), we know from  \cite{VV}  that a semi-parallel hypersurface $M^m$ in
$S^m\times\mathbb{R}$ is one of the as following:
(I) $m=2$ and $M^2$ is flat;\\
(II) $M$ is totally umbilical;\\
(III) $M$ is an open part of rotation hypersurface with $\lambda\mu=-\cos^2\alpha$, or\\
(IV) $M^m\subset \bar M^{m-1}\times\mathbb{R},$ where $\bar M$ is a
semi-parallel hypersurface of $S^m$, that is, $M^m$ is a vertical
cylinder.

By Theorem \ref{umbbih}, we obtain the conclusion for the case (II).
Therefore, we only need to consider the case (III).

The case (III):
Since $M$ is a rotation hypersurface, by (\ref{change}) and $\lambda\mu=-\cos^2\alpha$ we have
$uu'{\rm cot}s=u^2-1.$
Solving this equation yields
$u=\pm\sqrt{1+C\,\sec^2(s)}.$
However, this does not satisfy Equation \eqref{54}. Thus, the proof for the case of $S^m\times \r$ is complete.

By using the classification of semi-parallel hypersurfaces in $H^m\times \r$ given in \cite{CKV} and an argument similar to the above, we obtain the proof  for the case of $H^m\times \r$. 
\end{proof}
\begin{remark}
We remark that \eqref{BHEq2T} is equivalent to the Codazzi equation. Therefore, to show the Case (I), we do not need the tangential part of \eqref{BHEq}.
\end{remark}
\section{Biharmonic Rotation surfaces in $S^2\times\r$ and $H^2\times\r$}
In this section, we focus our attention on rotation surfaces in
$S^2\times\r$ and $H^2\times\r$. It should be remarked that we
choose in this section the parametrizations of rotation surfaces
developed in \cite{AEG}, which is different from the ones in Section
4. With this parametrizations, one could easily obtain some
classification results on biharmonic rotation surfaces in
$S^2\times\r$ and $H^2\times\r$.

We first derive the equivalent equations for a rotation surfaces in
$S^2\times\r$ to be biharmonic.
\begin{theorem}\label{To1}
A rotation surface  $f: \r^2\supseteq D^2\to S^2\times
\r$ with
 \begin{equation}\label{528}
 f(r, \theta)=(\sin k(r), \cos k(r) \cos \theta, \cos k(r) \sin \theta, h(r))
 \end{equation}
 is biharmonic if and only if it is minimal, an open part of the vertical cylinder $S^1(\frac{1}{\sqrt{2}})\times \r$ or
\begin{equation}\label{BRE}
\begin{cases}
\Delta^2\, k\, +2 \Delta\, k + 2(\sec^2 k\, \tan k)k'^2+(1-\tan^2 k)\tan k=0,\\
(k'H)'=C\sec k,\; {\rm for\; some\; constant} \;\;C\geq0,\\
h'^2+k'^2=1,
\end{cases}
\end{equation}
where  $\Delta$ is the Laplacian on the surfaces defined by the
induced metric.
\end{theorem}
\begin{proof}
Without lost of generality, we may assume that  the profile curve\\ $(\sin \rho(r),  \cos\rho(r) ,0, h(r))$ ($\cos \rho(r)\ge 0$) of the rotation surface defined by (\ref{528})  is parametrized by arclength parameter so that we have
 \begin{equation}\label{AL}
k'(r)^2+h'(r)^2=1.
 \end{equation}

 We choose geodesic polar coordinates $(\rho, \phi)$ on $S^2$ so that its metric takes the form $g_{S^2}=\rm d \rho^2+\cos^2 \rho \rm d \phi^2$, and the product metric on $N^3=S^2\times \r$ reads
 \begin{equation}\notag
 g_N=\rm d \rho^2+\cos^2 \rho\, \rm d \phi^2+\rm dt^2.
 \end{equation}

With the chosen local coordinates, the rotation surface can be
viewed as an isometric immersion
 \begin{eqnarray}\label{DW}
 f: M^2 &=&\{(r, \theta)\subset \r^2\}\to (S^2\times \r, \rm d \rho^2+\cos^2 \rho\, {\rm d} \phi^2+\rm dt^2),\\\notag
 f(r, \theta)&=&(k(r), \theta, h(r)),\; {\rm or}\; \rho(r, \theta)=k(r),\; \phi (r,\theta)=\theta,\; t (r, \theta)=h(r).
 \end{eqnarray}
  A straightforward computation yields
 \begin{eqnarray}\notag
{\rm d}f(\partial_r)=f_r=(k', 0, h'),\;\;{\rm
d}f(\partial_\theta)=f_{\theta}=(0, 1, 0),
 \end{eqnarray}
and the induced metric on the rotation surface given by
  \begin{equation}\notag
g_{M^2}={\rm d} r^2+\cos^2 k(r)\, {\rm d}\, \theta^2.
 \end{equation}

The computations of the principal and the mean curvatures of the
rotation surface can be done as follows.

 By identifying the point $(r,\theta)\in M^2$ with its image $f(r, \theta)\in S^2\times \r$ and the vector $X$ tangent to $M^2$ to the vector ${\rm d} f (X)$ tangent to $S^2\times \r$, we choose an orthonormal frame
 \begin{eqnarray} \notag
 e_1&=&k'\partial_{\rho}+h'\partial_t= \cos \alpha\, \partial _{\rho}+\sin \alpha\,\partial_t, \;\;
 e_2=\sec\rho\, \partial _{\phi}, \\\notag
 \xi &=&-h'\partial_{\rho}+k'\partial_t=-\sin \alpha\, \partial _{\rho}+\cos \alpha\,\partial_t
 \end{eqnarray}
 on the ambient space adapted to the rotation surface with $\xi$ being the unit normal vector field of the surface.

 Note that in the above, we have the angle function of the surface satisfying
 \begin{equation}\label{Ag1}
 \cos \alpha\,=\langle \xi, \partial_t\rangle=k',\;\;\; \sin
 \alpha\,=\sqrt{1-k'^2}=h'.
 \end{equation}

 A straightforward computation gives
 \begin{eqnarray}\notag
\langle [e_1, \xi], e_1\rangle = -h''/k',\;\;\; \langle [e_2,
\xi],e_2\rangle=h'\,\tan k.
 \end{eqnarray}
 A further computation using these and Koszul's formula, we have
  \begin{eqnarray}\notag
 \langle A_{\xi}\, e_1, e_1\rangle&=&-\langle \nabla^N_{e_1}\,\xi, e_1\rangle=-\langle [e_1, \xi], e_1\rangle=h''/k',\\\notag
  \langle A_{\xi}\, e_2, e_2\rangle&=&-\langle \nabla^N_{e_2}\,\xi, e_2\rangle=-\langle [e_2, \xi],e_2\rangle= -h'\,\tan k\,.
 \end{eqnarray}
 It follows that $e_1, e_2 $ are the two principal directions with the principal curvatures
 \begin{equation}\notag
 \lambda_1=h''/k',\;\;\;\lambda_2= -h'\,\tan k.
 \end{equation}
It follows that the mean curvature of the  rotation surface is given
by
\begin{equation}\label{KE}
H=\frac{1}{2} (\lambda_1+\lambda_2)=\frac{1}{2k'} (h''-h'k'\,\tan
k).
\end{equation}
 By Corollary \ref{O10}, the isometric immersion (\ref{DW}) is biharmonic if and only if both the height function
\begin{equation}\notag
h: (M^2, {\rm d} r^2+\cos^2 k(r)\, {\rm d}
\theta^2)\to\r,\;\; h(r, \theta)=h(r)
\end{equation}
and the map
\begin{eqnarray}\label{M2}
&& \left(M^2, {\rm d} r^2+\cos^2 k(r)\, {\rm d} \theta^2)
\right)\to (S^2, g_{S^2}={\rm d} \rho^2+\cos^2 \rho\,
{\rm d} \phi^2),\\\notag && \varphi(r, \theta)=(k(r), \theta)
\end{eqnarray}
are  biharmonic.

It was proved in \cite{WOY} (Corollary 2.3) that a rotationally
symmetric  map
$\varphi:(M^2,dr^2+\sigma^2(r)d\theta^2)\to
(N^2,d\rho^2+\lambda^2(\rho)d\phi^2)$, $\varphi(r,\theta)=(\rho(r),
\theta)$ is biharmonic if and only if it solves the system
\begin{equation}\notag
\begin{cases}
x''+\frac{\sigma'}{\sigma}x'-\frac{(\lambda\lambda')'(\rho)}{\sigma^2}x=0,\\
x=\tau^1=\rho''+\frac{\sigma'}{\sigma}\rho'-\frac{\lambda\lambda'(\rho)}{\sigma^2}.
\end{cases}
\end{equation}
Applying this, we conclude that the rotationally symmetric map
defined by (\ref{M2}) is biharmonic if and only if
\begin{equation}\label{Sl4}
\begin{cases}
x''-k' (\tan k)\, x'+(1-\tan^2 k)x=0,\\
x=\tau^1=k''+(1-k'^2)\tan k.
\end{cases}
\end{equation}
A straightforward computation gives the following lemma.
\begin{lemma}\label{LP}
For a function $u:(M^2,dr^2+\cos^2 k(r) d\theta^2)\to\r$
with $u(r, \theta)=u(r)$, we have
\begin{equation}
\Delta u= u''-k' (\tan k)  u'=u''+(\ln \cos k)' u'.
\end{equation}
In particular,
\begin{eqnarray}\label{I3}
\Delta h=h''-(k'\tan k) h'= 2k' H.
\end{eqnarray}
\end{lemma}
Using this, together with (\ref{Sl4}), and a further computation, we
conclude that the rotationally symmetric map (\ref{M2}) is
biharmonic if and only if
\begin{equation}\notag
\Delta^2\, k\, +2 \Delta\, k + 2(\sec^2 k\, \tan k)k'^2+(1-\tan^2
k)\tan k=0.
\end{equation}
Moreover, the biharmonicity of the height function $\Delta^2 h=0$
implies that
\begin{eqnarray}\label{HH1}
\Delta^2 h=(\Delta h)''-(k'\tan k) (\Delta h)'=0.
\end{eqnarray}
To solve equation (\ref{HH1}), we note that if $(\Delta h)'\neq0$ in
an open set.  Then we can solve Equation (\ref{HH1}) to have
\begin{align}\label{GDO19}
(\Delta h)'=C\sec k
\end{align}
 for some constant $C\ne0$.

 Now  if $(\Delta h)'\equiv 0$, then $\Delta h=C_1$, a constant. If $C_1\ne 0$, then the corresponding solutions can be included in (\ref{GDO19}) by allowing $C=0$. Otherwise, we have $\Delta h=0$, from which and (\ref{I3}),  we have either $H=0$ and the rotation
surface is minimal, or $k'=0$. The latter case $k'=0$ means that
$k={\rm constant}$ in an open set. Combining this and (\ref{GD10})
we have
\begin{eqnarray}\notag
\cos k =\pm \frac{1}{\sqrt{2}}, \:{\rm and\: hence},\;\;\sin k =\pm
\frac{1}{\sqrt{2}}.
\end{eqnarray}
Substituting $k={\rm constant}$ into (\ref{AL}) we have $h'(r)=1$
and hence $h(r)=r+r_0$. Therefore, the biharmonic rotation surface
$f: \r^2\supseteq D^2\to S^2\times \r$ is given by
 \begin{equation}\notag
 f(r, \theta)=(\pm \frac{1}{\sqrt{2}}, \pm \frac{1}{\sqrt{2}} \cos \theta, \pm \frac{1}{\sqrt{2}} \sin \theta, r+r_0),
 \end{equation}
which are exactly the vertical cylinder.

Putting all the results together we complete the proof of the
theorem.
\end{proof}
\begin{theorem}\label{To2}
A rotation surface  $f: \r^2\supseteq D^2\to S^2\times
\r$ with
 \begin{equation}\notag
 f(r, \theta)=(\sin k(r), \cos k(r) \cos \theta, \cos k(r) \sin \theta, h(r))\nonumber
 \end{equation}
and $\Delta k=0$ is  biharmonic if and only if it is minimal or an
open subset of the vertical cylinder $S^1(\frac{1}{\sqrt{2}})\times
\r\subset S^2\times \r$.
\end{theorem}
\begin{proof}
If $\Delta k=0$, then the the first equation of  (\ref{BRE}) reduces
to
\begin{equation}\notag
 2(\sec^2k\, \tan k)k'^2+(1-\tan^2 k)\tan k=0.
\end{equation}
It follows that either we have $\tan k\equiv 0$, in which case, the
$x=\Delta k +\tan k=0$ and hence tension field vanishes identically
and the surface is minimal, or, in an open set in which $\tan k\ne
0$, we have
\begin{equation}\label{GD10}
 k'^2=\frac{1}{2}(1-2\cos^2k).
\end{equation}
On the other hand,  $\Delta k=0$ is equivalent to
\begin{equation}\label{LP2}
k''+(\ln \cos k)' k'=0.
\end{equation}
It follows that either (i) $k'\equiv 0$, or (ii) there is an open
set in which $k'\ne 0$. Now we show that (ii) is impossible.  In
fact, if $k'\ne 0$, we use $\Delta k=0$ to have $(\ln (k' \cos k))'
=0$, which is equivalent to $k'=\frac{C}{ \cos k}$, for some
non-zero constant. Substituting this into (\ref{GD10}) we have
\begin{equation}\label{GD11}
\frac{C^2}{ \cos^2 k}=\frac{1}{2}(1-2\cos^2k),
\end{equation}
which is equivalent to
\begin{equation}\label{GD11}
-2\cos^4k+\cos^2k-2C^2=0.
\end{equation}
It is easy to see that the quadratic equation (\ref{GD11}) in
$\cos^2k$ either has no solution or has constant solution $\cos^2
k=C_1$. In this case, $k$ is constant and hence $k'=0$ in the open
set. This contradicts the assumption that $k'\ne 0$ in the open set.
It following that the only solutions of  (\ref{GD10}) and
(\ref{LP2}) is $k'\equiv 0$. Therefore, the biharmonic rotation
surface is given by the vertical cylinder. Thus, we obtain the
theorem.
\end{proof}
Flat rotation hypersurfaces in $S^m\times \r$ and $H^m\times \r$ were characterized by the expressions of their profiles. In the following we give a classification of flat rotation surfaces in $S^2\times \r$ and $H^2\times \r$.
\begin{theorem}\label{To2}
A rotation surface  $f: \r^2\supseteq D^2\to S^2\times
\r$ with
 \begin{equation}
 f(r, \theta)=(\sin k(r), \cos k(r) \cos \theta, \cos k(r) \sin \theta, h(r))\nonumber
 \end{equation}
and $K=0$ is  biharmonic if and only if it is minimal or an open
subset of the vertical cylinder $S^1(\frac{1}{\sqrt{2}})\times
\r\subset S^2\times \r$.
\end{theorem}
\begin{proof}
It follows from Gauss equation that the Gauss curvature $K$ is given
by
 \begin{equation}
 K=\lambda_1\lambda_2+\cos^2\alpha.\nonumber
 \end{equation}
It follows that if $K=0$, then  \eqref{KE} and \eqref{Ag1}
apply to give 
\begin{equation}\label{525}
\lambda_1\lambda_2+\cos^2\alpha=k''\tan k+k'^2=0,
\end{equation}
where in obtaining the equation we have also used $h'h''=-k'k''$ which follows from
(\ref{AL}). Equation (\ref{525}) can be rewritten as
  \begin{eqnarray}
k''\sin k+k'^2\cos k=(k'\sin k)'=-(\cos k)''=0,\nonumber
 \end{eqnarray}
which can be solved to give
 \begin{eqnarray}\label{Ar}
   \cos k=Ar+B 
 \end{eqnarray}
 for some constant $A$ and $B$.
If $A=0$, then $k$ is a constant and hence, by \eqref{Ag1}, 
the angle function $\alpha$ is constant. Thus, by the surface is an
open subset of the vertical cylinder $S^1(\frac{1}{\sqrt{2}})\times
\r\subset S^2\times \r$. If $A\neq0$, we use (\ref{Ar}) to have
 \begin{align}\notag
k'=&-\frac{A}{\sin k}\label{FK2},\;\;\;\Delta k=k''-k'^2\tan k=-\frac{A^2}{\sin^3 k\cos k}\\\notag
\Delta^2 k=&-\frac{A^4}{\sin^7 k\cos^3 k}(16\cos^4 k-2\cos^2 k+1)
 \end{align}
Substituting the above equations  into the first equation of \eqref{BRE}, we have
 \begin{eqnarray}
2\cos^{10}k-9\cos^8k-4(A^2-4)\cos^6 k-2(8A^4-5A^2+7)\cos^4
k\nonumber\\
+2(A^4-4A^2+3)\cos^2k-A^4+2A^2-1=0.\label{FK5}
 \end{eqnarray}
Taking into account $\cos k=Ar+B$, we see that Equation \eqref{FK5} is a
non-trivial polynomial equation $p(r)\equiv 0$ of  degree $10$.
So, all coefficients of the polynomial $p(r)$ should be
zero. In particular, the leading term gives $2A^{10}=0$, which
contradicts our assumption that $A\neq0$. This completes the proof of
the theorem.
\end{proof}

Similar arguments apply to give the following results on biharmonic  rotation
surfaces in $ H^2\times \r$.
\begin{theorem}\label{H1}
A rotation surface  $ \psi: \r^2\supseteq D^2\to
H^2\times \r$ with
 \begin{equation}\notag
 \psi(r, \theta)=( \cosh k(r), \sinh k(r)\cos \theta, \sinh k(r)\sin \theta, h(r))
 \end{equation}
 is biharmonic if and only if it is minimal or
\begin{equation}\notag
\begin{cases}
\Delta^2\, k\, -2 \Delta\, k - 2\coth k\,(\coth^2 k-1)\,k'^2+(1+\coth^2 k)\coth k=0,\\
(k'H)'=\frac{C}{\sinh k},\; {\rm for\; some\; constant} \;\;C,\\
h'^2+k'^2=1,
\end{cases}
\end{equation}
where  $\Delta$ is the Laplacian on the surfaces defined by the
induced metric.
\end{theorem}
\begin{theorem}\label{H2}
A rotation surface  $f: \r^2\supseteq D^2\to H^2\times
\r$,
 \begin{equation}
 \psi(r, \theta)=( \cosh k(r), \sinh k(r)\cos \theta, \sinh k(r)\sin \theta, h(r))\nonumber
 \end{equation}
with (i) $\Delta k=0$ or (ii) $K=0$  is biharmonic if and only if it is minimal.
\end{theorem}

\end{document}